\documentclass[a4paper,12pt]{amsart}

\usepackage{amsfonts}
\usepackage{amsmath}
\usepackage{amssymb}

\usepackage{mathrsfs}
\usepackage[colorlinks]{hyperref}

\newcommand{\M}{\mathcal{M}}

\newcommand{\C}{\mathbb{C}}
\newcommand{\N}{\mathbb{N}}

\renewcommand{\bar}{\overline}
\renewcommand{\hat}{\widehat}

\def\calL{{\mathcal L}}

\numberwithin{equation}{section}

\newtheorem{thm}{Theorem}[section]

\newtheorem{lem}[thm]{Lemma}
\newtheorem{prop}[thm]{Proposition}

\theoremstyle{remark}
\newtheorem{rem}[thm]{Remark}
\newtheorem{defn}{Definition}[section]
\newtheorem{ex}[thm]{Example}

\newcommand{\Del}[1]{}

\newcommand{\HS}{\mathtt{HS}}

\def\calH{{\mathcal H}}

\def\N{\mathbb{N}}

\begin{document}

\title[Sequence spaces of eigenfunction expansions on compact manifolds]
{Eigenfunction expansions of ultradifferentiable functions and ultradistributions. II. Tensor representations.}

\author[Aparajita Dasgupta]{Aparajita Dasgupta}
\address{
  Aparajita Dasgupta:
  \endgraf
  \'Ecole polytechnique f\'ed\'erale de Lausanne
  \endgraf
  Facult\'e des Sciences
  \endgraf
  CH-1015 Lausanne 
  \endgraf
  Switzerland
  \endgraf
  {\it E-mail address} {\rm aparajita.dasgupta@epfl.ch}
  }

\author[Michael Ruzhansky]{Michael Ruzhansky}
\address{
  Michael Ruzhansky:
  \endgraf
  Department of Mathematics
  \endgraf
  Imperial College London
  \endgraf
  180 Queen's Gate, London SW7 2AZ 
  \endgraf
  United Kingdom
  \endgraf
  {\it E-mail address} {\rm m.ruzhansky@imperial.ac.uk}
  }

\thanks{The second
 author was supported by the EPSRC Grants EP/K039407/1 and 
EP/R003025/1, and by 
 the Leverhulme Research Grant RPG-2017-151. No new data was collected or generated during the course of this research.}
\date{\today}


\begin{abstract}
In this paper we analyse the structure of the spaces of coefficients of eigenfunction expansions of functions in Komatsu classes on compact manifolds, continuing the research in our paper \cite{DaR2}. We prove that such spaces of  Fourier coefficients are perfect sequence spaces. As a consequence we describe the tensor structure of sequential mappings on spaces of Fourier coefficients and characterise their adjoint mappings. In particular, the considered classes include spaces of analytic and Gevrey functions, as well as spaces of ultradistributions, yielding tensor representations for linear mappings between these spaces on compact manifolds.
\end{abstract}

\maketitle

\tableofcontents

\section{Introduction}

The present paper is a continuation of our paper \cite{DaR2} where we have characterised Komatsu spaces of ultradifferentiable functions and ultradistributions  on compact manifolds in terms of the eigenfunction expansions related to positive elliptic operators. In particular, these classes include the spaces of analytic, Gevrey and smooth functions as well as the corresponding dual spaces of distributions and ultradistributions, in both Roumieu and Beurling settings. 

In particular, this extended the earlier characterisations of analytic functions on compact manifolds in terms of the eigenfunction expansions by Seeley \cite{see:exp} (see also \cite{see:ex}), and characterisations of Gevrey spaces and ultradistributions on tori \cite{Tag1} and on compact Lie groups and homogeneous spaces \cite{DaR1}.

For example, if $E$ is a positive elliptic pseudo-differential operator on a compact manifold $X$ without boundary and $\lambda_j$ denotes its eigenvalues in the ascending order, then smooth functions on $X$ can be characterised in terms of their Fourier 
coefficients:
\begin{equation}\label{EQ:smooth}
f\in C^{\infty}(X) \; \Longleftrightarrow \;
\forall N\; \exists C_{N}: 
|\widehat{f}(j,k)|\leq C_{N} \lambda_{j}^{-N} 
\textrm{ for all } j\geq 1, 1\leq k\leq d_{j},
\end{equation}
where $\hat{f}(j,k)= \left(f, e^{k}_j\right)_{L^2}$ with $e_j^k$ being the $k^{th}$ eigenfunction corresponding to the eigenvalue $\lambda_j$ (of multiplicity $d_j$), see \eqref{EQ:FC}.
If $X$ and $E$ are analytic, the result of Seeley \cite{see:exp} can be reformulated 
as
\begin{equation}\label{EQ:analytic}
f \textrm{ is analytic } \Longleftrightarrow 
 \exists L>0\; \exists C: 
|\widehat{f}(j,k)|\leq C e^{-L\lambda_j^{1/\nu}} 
\textrm{ for all } j\geq 1, 1\leq k\leq d_{j},
\end{equation}
where $\nu$ is the order of the pseudo-differential operator $E$.
In \cite{DaR2} we extended such characterisations to Gevrey classes and, more generally, to Komatsu classes of ultradifferentiable functions and the corresponding classes of ultradistributions.

In this paper we continue this analysis showing that the appearing spaces of coefficients with respect to expansions in eigenfunctions of positive elliptic operators are perfect spaces in the sense of the theory of sequence spaces (see e.g. K{\"o}the \cite{Kothe:BK-top-vector-spaces-I}).
Consequently, we obtain tensor representations for linear mappings between spaces of ultradifferentiable functions and the corresponding spaces of ultradistributions. Such discrete representations in a given basis are useful in different areas of time-frequency analysis, in partial differential equations, and in numerical investigations. Due to possible multiplicities of eigenvalues the mappings beget the tensor structure rather than the matrix one as it would be in the case of simple eigenvalues, and our results are new for both situations.

Our analysis is based on the global Fourier analysis on compact manifolds which was consistently developed in \cite{DR}, with a number of subsequent applications, for example to the spectral properties of operators \cite{Delgado-Ruzhansky:JFA-2014}, or to the wave equations for the Landau Hamiltonian \cite{RT:LMP}. The corresponding version of the Fourier analysis based on expansions with respect to biorthogonal systems of eigenfunctions of non-self-adjoint operators has been developed in \cite{RT:IMRN}, with a subsequent extension in \cite{RT:MMNP}.

The obtained characterisations of Komatsu classes found their applications, for example for the well-posedness problems for weakly hyperbolic partial differential equations \cite{Garetto-Ruzhansky:wave-eq}. The spaces of coefficients of eigenfunction expansions in ${\mathbb R}^n$ with respect to the eigenfunctions of the harmonic oscillator have been analysed in \cite{GPR} , and the corresponding Komatsu classes have been investigated in \cite{VV}. The original Komatsu spaces of ultradifferentiable functions and ultradistributions have appeared in the works \cite{KO1, KO2, KO3} by Komatsu (see also Rudin \cite{Rudin:bk-RandCanalysis-1974}), extending the original works by Roumieu \cite{Roumieu:1962}. The universality of the spaces of Gevrey functions on the torus has been established in \cite{Tag2}.

The regularity properties of spaces of distributions and ultradistributions have been analysed in \cite{Pilipovic-Scarpalezos:PAMS-2001}, and their convolution properties appeared in \cite{Pilipovic-Prangoski:Roumieu-MM-2014}.

The characterisations in terms of the eigenfunction expansions provide for descriptions alternative to those using the classical Fourier analysis, with applications in the theory of partial differential equations, see e.g. \cite{Rodino:bk-Gevrey}.
For some other applications of this type of analysis one can see e.g. \cite{Carmichael-Kamiski-Pilipovic:BK,Delcroix-Hasler-Pilipovic:periodic}.

\smallskip
The paper in organised as follows. In Section \ref{SEC:Fourier} we briefly recall the constructions leading to the global Fourier analysis on compact manifolds. In Section \ref{SEC:seqspaces} we very briefly recall the relevant definitions from the theory of sequence spaces.
In Section \ref{SEC:Komatsu}  we present the main results of this paper and their proofs.
In Section \ref{SEC:Beurling} we first recall the definitions for Beurling version of the spaces and then give the statement of the corresponding adjointness Theorem \ref{THM:adj} in this  case.

\smallskip
In this paper we adopt the notation $\N_0=\N\cup\{0\}$.

\section{Fourier analysis on compact manifolds}
\label{SEC:Fourier}

Let $X$ be a closed  $C^{\infty}$-manifold of dimension $n$  endowed with a fixed measure $dx.$ We first recall an abstract statement from \cite[Theorem 2.1]{DR} giving rise to the Fourier analysis on $L^2(X)$.

\begin{thm}\label{THM:DR-inv}
Let $\calH$ be a complex Hilbert space and let $\mathcal{H^{\infty}}\subset \mathcal H$ be a dense linear subspace of $\mathcal H.$ Let $\{d_j\}_{j\in\mathbb{N}_0}\subset \mathbb{N}$ and let $\{e^{k}_{j}\}_{{j\in\mathbb{N}_{0}, 1\leq k\leq d_j}}$ be an orthonormal basis of $\mathcal H$  such that $e^{k}_j\in\mathcal{H}^{\infty}$ for all $j$ and $k$. Let $H_j:={\rm span}\{e^{k}_{j}\}_{k=1}^{d_j},$ and let $P_j:\calH\rightarrow H_j$ be the orthogonal projection. For $f\in\calH,$ we denote $\hat{f}(j,k):=(f,e^{k}_j)_{\calH}$ and let $\hat f(j)\in \mathbb{C}^{d_j}$ denote the column of $\hat f(j,k),$  $1\leq k\leq d_j.$ Let $T: \calH^{\infty}\rightarrow \calH$ be a linear operator. Then the following conditions (i)-(iii) are equivalent.
\begin{enumerate}
\item For each $j\in\mathbb N_{0},$ we have $T(H_j)\subset H_j.$
\item For each $l\in\mathbb{N}_0$ there exists a matrix $\sigma(l)\in\mathbb{C}^{d_l\times d_l}$ such that for all $e^{k}_j$,
$$\hat{Te^{k}_{j}}(l,m)=\sigma(l)_{mk}\delta_{jl}.$$
\item If in addition all $e_{j}^{k}$ are in the domain of $T^{*}$, then for each $l\in\mathbb{N}_0$ there exists a matrix $\sigma(l)\in\mathbb{C}^{d_l\times d_l}$ such that for all $f\in{\calH}^{\infty} $ we have
$$\hat{Tf}(l)=\sigma(l)\hat{f}(l).$$
The matrices in (ii) and (iii) coincide.

The equivalent properties (i)--(iii) follow from the condition:
\item For each $j\in\mathbb{N}_0,$ we have $TP_j=P_jT$ on $\calH^{\infty.}$

If, in addition, $T$ extends to a bounded operator $T\in\calL(H)$ then (iv) is equivalent to (i)--(iii).
\end{enumerate}
\end{thm}

Under the assumptions of Theorem \ref{THM:DR-inv} we have the direct sum decomposition 
$$
\calH =\oplus_{j=0}^{\infty}H_j, \quad H_j=\textrm{ span }\{e^{k}_j\}_{k=1}^{d_j},
$$ 
and we have $d_j=\dim H_j.$ Here we will consider $\calH=L^2(X)$ for a compact manifold $X$ with $H_j$ being the eigenspaces of an elliptic positive pseudo-differential operator $E.$  

The eigenvalues of $E$ (counted without multiplicities) form a sequence ${\lambda_j}$, $j\in\N$, which we order so that 
$$0=:\lambda_0<\lambda_1<\lambda_2<...$$
For each eigenvalue $\lambda_j,$ there is the corresponding finite dimensional eigenspace $H_j$ of functions on $X,$ which are smooth due to the ellipticity of $E.$ We set 
$$d_0:=\dim H_0, \quad H_0:=\ker E, ~~\lambda_0:= 0.$$
Since the operator $E$ is elliptic, it is Fredholm, hence also $d_0<\infty.$  

We denote by $\Psi^{\nu}_{+e}(X)$ the space of positive elliptic pseudo-differential operators on order $\nu>0$ on $M$.
Here we recall a useful relation between the sequences $\lambda_j$ and $d_j$ of eigenvalues of $E\in \Psi^{\nu}_{+e}(X)$ and their multiplicities from \cite{DR}.

\begin{prop}\label{PROP:dlambdas}
Let $X$ be a closed manifold of dimension $n$, and let $E\in \Psi^{\nu}_{+e}(X)$, with $\nu>0.$ Then there exists a constant $C>0$ such that we have
$$d_j\leq C(1+\lambda_j)^{\frac{n}{\nu}}$$ for all $j\geq 1.$ Moreover, we also have 
$$\sum^{\infty}_{j=1}d_j(1+\lambda_j)^{-q}<\infty \;\textrm{ if and only if } \quad  q>\frac{n}{\nu}.$$ 
\end{prop}

For $f\in L^{2}(X),$ by definition we have the Fourier series decomposition
$$f=\sum_{j=0}^{\infty}\sum_{k=1}^{d_j}\hat{f}(j,k)e^{k}_{j}.$$  
The Fourier coefficients of $f\in L^2(X)$ with respect to the orthonormal basis $\{e^{k}_j\}$ are denoted by
\begin{equation}\label{EQ:FC}
\mathcal{F}f(j,k)=\hat{f}(j,k):= \left(f, e^{k}_j\right)_{L^2}.
\end{equation} 
We denote the space of Fourier coefficients by
\begin{equation}\label{EQ:sigma}
 \Sigma=\{v=(v_l)_{l\in\mathbb{N}_{0}},~ v_{l}\in\mathbb{C}^{d_l}\}.
\end{equation} 
Since $\{e^{k}_j\}_{j\geq 0}^{1\leq k\leq d_j}$ is a complete orthonormal system of  $L^{2}(X)$ we have the Plancherel formula 
  $$||f||^{2}_{L^2(X)}=\left(\sum_{j=0}^{\infty}\sum_{k=1}^{d_j}|\hat{f}(j,k)|^{2}\right)^{1/2}=||\hat f||^{2}_{l^{2}(\mathbb N_{0},\Sigma)}=:\sum_{j=0}^{\infty}||\hat{f}(j)||^{2}_{\mathtt{HS}},$$
  where we interpret $\hat f$ as an element of the space 
  $$l^{2}(\mathbb N_{0},\Sigma)=\left\{h: \mathbb N_{0}\rightarrow \prod_{j}\mathbb{C}^{d_j}: h(j)\in \mathbb{C}^{d_j} , \sum_{j=0}^{\infty}\sum_{k=1}^{d_j}|{h}(j,k)|^{2}<\infty \right\}.$$
 We endow $l^{2}(\mathbb N_{0},\Sigma)$ with the norm 
 $$||h||_{l^{2}(\mathbb N_{0},\Sigma)}=
 \left(\sum_{j=0}^{\infty}\sum_{k=1}^{d_j}|{h}(j,k)|^{2}\right)^{1/2}.$$\\
We can think of $\mathcal{F}=\mathcal{F}_{X}$ as of the Fourier transform which is an isometry form $L^{2}(X)$ onto $l^{2}(\mathbb N_{0},\Sigma).$ The inverse of this Fourier transform can be then expressed by
$$(\mathcal{F}^{-1}h)(x)=\sum_{j=0}^{\infty}\sum_{k=1}^{d_j}h(j,k)e^{k}_{j}(x).$$
If $f\in L^{2}(X)$ we can write
\begin{equation} 
\hat{f}(j)= \begin{pmatrix}
    \hat{f}(j,1)\\
    \vdots\\
    \vdots\\
    \vdots\\
    \vdots\\
    \hat{f}(j,d_j)
  \end{pmatrix} \in \mathbb C^{d_j},
  \end{equation}
  thus thinking of the Fourier transforn always as a column vector.  In particular,  we think  of 
$$\hat{e^{k}_{j}}(l)=\left(\hat{e^{k}_{j}}(l,m)\right)_{m=1}^{d_l}$$ as of a column, and we notice that 
$$\hat{e^{k}_{j}}(l,m)=\delta_{jl}\delta_{km}.$$

  \section{Sequence spaces and sequential linear mappings}
  \label{SEC:seqspaces}
  
  We briefly recall that a sequence space $E$ is a linear subspace of 
$$\mathbb{C}^{\mathbb Z}=\{a=(a_j)|a_j\in\mathbb{C}, j\in \mathbb{Z}\}.$$
The dual $\hat{E}$ ($\alpha$-dual in the terminology of G. Kothe \cite{Kothe:BK-top-vector-spaces-I}) is a sequence  space defined by
$$\hat{E}=\{a\in \mathbb{C}^{\mathbb Z}: \sum_{j\in \mathbb{Z}} |u_j|||a_j|<\infty
\textrm{ for all }u\in E\}.$$

A sequence space $E$ is called {\em perfect} if $\hat{\hat{E}}=E$.
A sequence space is called {\em normal} if $u=(u_j)\in E$ implies $|u|=(|u_j|)\in E.$
A dual space $\hat{E}$ is normal so that any perfect space is normal.

A pairing ${\langle\cdot,\cdot\rangle}_{E}$ on $E$ is a bilinear function on $E\times\hat{E}$ defined by 
$$\langle u,v\rangle_{E}=\sum_{j\in \mathbb{Z}}{u_jv_j}\in\mathbb{C},$$ 
which converges absolutely by the definition of $\hat{E}.$

\begin{defn} $\phi: E\rightarrow \mathbb{C}$ is called a {\em sequential linear functional} if there exists some $a\in\hat{E}$ such that $\phi(u)=\langle u,a\rangle_E$ for all $u\in E.$ We abuse the notation by also writing $a: E\rightarrow \mathbb{C}$ for this mapping.\end{defn}

\begin{defn} A mapping $f:E\rightarrow F$ between two sequence spaces is called a {\em sequential linear mapping} if 
\begin{enumerate}
\item $f$ is algebraically linear,
\item for any $v\in \hat F,$ the composed mapping $v\circ f\in\hat{E}.$ 
\end{enumerate}
\end{defn}
  
\section{Tensor representations for Komatsu classes and their $\alpha$-duals}
\label{SEC:Komatsu}

Let $M_{k}$ be a sequence of positive numbers such that 

\medskip
\noindent
(M.0) $M_0=1$,\\
(M.1) (Stability) $M_{k+1}\leq AH^{k}M_k, $  $k=0,1,2,\ldots,$\\
(M.2) $M_{2k}\leq AH^{k}\min_{0\leq q\leq k} M_qM_{k-q},$  $k=0,1,2,...,$ for some $A, H>0$.

\medskip
In a sequence of papers \cite{KO1,KO2,KO3} Komatsu investigated classes of ultradifferentiable functions on ${\mathbb{R}}^{n}$ associated to the sequence ${M_k}$, namely, the spaces of functions $\Psi\in C^{\infty}(\mathbb{R}^{n})$ such that  for every compact 
$K\subset\mathbb{R}^{n}$ there exist $h>0$ and a constant $C>0$ such that
\begin{equation}\sup_{x\in K}|\partial^{\alpha}\psi(x)|\leq Ch^{|\alpha|}M_{|\alpha|}
\end{equation}
holds for all multi-indices $\alpha\geq 0$.
Similar to the case of usual distributions given a space of ultradifferentiable functions satisfying (4.1) we can define a space of ultradistributions as its dual.
 
We now recall the analogous definition of the Komatsu ultradifferentiable functions $\Gamma_{\{\M_k\}}(X)$ and its $\alpha$-dual $\left[\Gamma_{\{\M_k\}}(X)\right]^{\wedge}$.
Here, as before, $X$ is a compact manifold without boundary and 
$E\in\Psi_{+e}^{\nu}(X)$ with $\nu>0$.

\begin{defn} The class $\Gamma_{\{M_k\}}(X)$ is the space of $C^{\infty}$ functions $\phi$ on $X$ such that there exist $h>0$ and $C>0$ such that we have
$$||E^{k}\phi||_{L^2(X)}\leq Ch^{\nu k}M_{\nu k},\quad k=0,1,2,\ldots,$$ 
where $\nu\in\mathbb{N}$ is the order of the positive elliptic pseudo-differential operator $E.$\end{defn}
  In \cite{DaR2} we have characterised the class $\Gamma_{\{M_k\}}(X)$ in terms of the eigenvalues of the operator $E$. 
 We assume that 
 
 \medskip
 \noindent
 (M.3) \quad For some $l,C_{l}>0$ we have $k!\leq C_{l} l^{k}M_{k}$, for all $k\in\mathbb{N}_{0}.$
 
\medskip

In the sequel,  for $w_l\in{\mathbb{C}}^{d_l}$ we write 
$$||w_l||_{\mathtt{HS}}:=\left(\sum_{j=1}^{d_l}\left|(w_l)_j\right|^{2}\right)^{1/2}.$$
  
\begin{thm}[{\cite{DaR2}}]\label{THM:gamma}
Assume conditions (M.0), (M.1), (M.2), (M.3). Then $\phi\in\Gamma_{\{M_k\}}(X)$ if and only if there exist constants $C>0$ and $L>0$ such that
$$||\hat{\phi}(l)||_{\mathtt{HS}}\leq C\exp\{-M(L\lambda_{l}^{1/\nu})\}, \quad \textrm{ for all } l\geq 1,$$
where $$M(r):=\sup_{k\in\N_0}\log\frac{r^{\nu k}}{M_{\nu k}}.$$
\end{thm}

\begin{ex}
As an example, for the (Gevrey-Roumieu) class of ultradifferentiable functions
$$\gamma^{s}(X)=\Gamma_{\{(k!)^{s}\}}(X),\quad 1< s<\infty,$$ 
we have $M(r)\simeq r^{1/s}$. This is also true for $s=1$, characterising the class of analytic functions if the manifold is analytic. The class $\gamma^{s}(X)$ coincides with the usual Gevrey class of functions on a manifold $X$ defined in terms of their localisations.
\end{ex}

Based on Theorem \ref{THM:gamma} we can  then write
 \begin{multline*}
 \Gamma_{\{\M_k\}}(X)=\left\{[\hat{\phi}(l)]_{l\in{\mathbb{N}}_{0}}:  \phi\in C^{\infty}(X),\; \exists C>0  \textrm{ such that } \right.  \\ \left.
  ||\hat{\phi}(l)||_{\mathtt{HS}}\leq C\exp\{-M(L\lambda_l^{1/\nu})\} , \forall l \geq 1 \right\}.
 \end{multline*}
For  $\phi\in  \Gamma_{\{\M_k\}}(X)$ we will write $\phi\approx \left[\hat{\phi}(l)\right]_{l\in{\mathbb{N}}_{0}}$  so that $\Gamma_{\{\M_k\}}(X)$ can be thought of as a sequence space, but it will be convenient to view it as a subspace of $\Sigma$ defined in \eqref{EQ:sigma}, taking into account the dimensions of the eigenspaces of the operator $E$.

Next we recall the definition of the $\alpha$-dual of the space $\Gamma_{\{\M_k\}}(X)$  (following \cite{DaR2}).\\

The $\alpha$-dual of the space $\Gamma_{\{\M_k\}}(X)$ of ultradifferentiable functions, denoted by $[\Gamma_{\{\M_k\}}(X)]^{\wedge},$ is given by
$$\left\{v=(v_l)_{l\in\mathbb{N}_0}\in \Sigma, v_l\in{\mathbb C}^{d_l}: \sum_{l=0}^{\infty}\sum_{j=1}^{d_l}|(v_l)_j||\hat{\phi}(l,j)|<\infty, \textrm{ for all } \phi\in \Gamma_{\{\M_k\}}(X)  \right\}.$$
 We also recall the following characterisations of the $\alpha$-duals established in \cite{DaR2}.
 
\begin{thm}\label{THM:gammahat}
Assume conditions (M.0), (M.1), (M.2), (M.3). The following statements are equivalent
\begin{enumerate}
\item $v\in [\Gamma_{\{\M_k\}}(X)]^{\wedge}$;
\item for every $L>0$ we have 
$$\sum_{l=0}^{\infty}\exp\left(-M(L\lambda_l^{1/\nu})\right)||v_l||_{\mathtt{HS}}<\infty;$$
\item for every $L>0$ there exists $K_L>0$ such that
$$||v_l||_{\HS}\leq K_{L}\exp\left(M(L\lambda_l^{1/\nu})\right)$$ holds for all $l\in\mathbb{N}_0.$
\end{enumerate}
\end{thm}

We will now show that the space $\Gamma_{\{\M_k\}}(X)$ is perfect.   In the
proof as well as in further proofs the following estimate will be useful:
\begin{equation}\label{EQ:weyllaw}
||e_l^{j}||_{L^{\infty}(X)}\leq C\lambda^{\frac{n-1}{2\nu}}_{l} \;\textrm{ for all }\;l\geq 1.
\end{equation} 
This estimate follows, for example, from the local Weyl law \cite[Theorem 5.1]{Hor}, see also \cite[Lemma 8.5]{DR}.

\begin{thm}\label{P:perfect}
Let $X$ be a compact manifold and assume conditions (M.0), (M.1), (M.2), (M.3). Then
$\Gamma_{\{\M_k\}}(X)$ is a perfect space, that is,  we have 
$$\Gamma_{\{\M_k\}}(X)=\left[\hat{\Gamma_{\{\M_k\}}(X)}\right]^{\wedge},$$ 
where 
$$\left[\hat{\Gamma_{\{\M_k\}}(X)}\right]^{\wedge}=\left\{w=(w_l)_{l\in{\mathbb{N}}_{0}}\in \Sigma: \sum^{\infty}_{l=0}\sum_{j=1}^{d_l}\left|(w_l)_j\right|\left|(v_l)_j\right|<\infty, \forall v \in \left[\Gamma_{\{\M_k\}}(X)\right]^{\wedge}\right\}.$$
\end{thm}

To prove this we first establish the following lemma:

\begin{lem} \label{L:perfect1}
We have $w\in \left[\hat{\Gamma_{\{\M_k\}}(X)}\right]^{\wedge}$ if and only if  there exists $L>0$ such that
$$\sum_{l=0}^{\infty}\exp\left(M(L\lambda_l^{1/\nu})\right)||w_l||_{\mathtt{HS}}<\infty.$$ 
\end{lem}

\begin{proof}[Proof of Lemma \mbox{\ref{L:perfect1}}] 
$\Longrightarrow$: 
For $L>0$ we consider the echelon space
$$
D_{L}=\left\{v=(v_{l})\in\Sigma: \exists C>0: |(v_l)_j|\leq C\exp(M(L\lambda_{l}^{1/\nu})),\forall 1\leq j\leq d_l\right\},
$$ 
where $\Sigma=\{v=(v_l)_{l\in\mathbb{N}_{0}},~ v_{l}\in\mathbb{C}^{d_l}\}$
is as in \eqref{EQ:sigma}.

By the diagonal transform we have 
$D_{L}\cong l^{\infty},$ and since $l^{\infty}$ is a perfect space so we have $\widehat{D_L}\cong l^{1},$ and it is given by
$$\widehat{D_L}=\left\{w=(w_l)\in\Sigma:\sum_{l=0}^{\infty}\sum_{j=1}^{d_l}\exp(M(L\lambda_{l}^{1/\nu}))|(w_l)_j|<\infty\right\}.$$
By Theorem \ref{THM:gammahat} we know that $\widehat{\Gamma_{\{M_k\}}(X)}=\cap_{L>0}D_L,$ and hence $\left[\widehat{\Gamma_{\{M_k\}}(X)}\right]^{\wedge}=\cup_{L>0}\widehat{D_L}.$
This means that $w\in \left[\widehat{\Gamma_{\{M_k\}}(X)}\right]^{\wedge}$ if and only if there exists $L_2>0,$ such that we have 
$$\sum_{l=0}^{\infty}\sum_{j=1}^{d_l}\exp(M(L_2\lambda_{l}^{1/\nu}))|(w_l)_j|<\infty.$$
Let $1\leq p< q\leq\infty$ and let $a\in\C^{d\times d}.$ Then we have the estimates
\begin{equation}\label{EQ:in}
 \|a\|_{\ell^p(\C)}\leq d^{2\left(\frac1p-\frac1q\right)}\|a\|_{\ell^q(\C)}
 \quad\textrm{ and } \quad
 \|a\|_{\ell^q(\C)}\leq d^{\frac{2}{q}}\|a\|_{\ell^p(\C)},
 \end{equation}
 see e.g. \cite[Lemma 3.2]{DaR1} for a simple proof.
In particular, we have $d^{-1}||a||_{l^{1}}\leq ||a||_{\mathtt{HS}}\leq d ||a||_{l^{1}}$  for $a\in\mathbb{C}^{d\times d}.$ Here we also note the estimate: for every $q,$ $L>0$ and $\delta>0$ there exists $C>0$ such that 
\begin{equation}\label{EQ:estlambdas}
\lambda_{l}^{q}e^{-\delta M\left(L\lambda_{l}^{1/\nu}\right)}\leq C,
\end{equation} 
see e.g. \cite[(2.15)]{DaR2}.
These estimates and Proposition \ref{PROP:dlambdas} imply
\begin{eqnarray}
& & \sum_{l=0}^{\infty}\exp(M(L\lambda_{l}^{1/\nu}))||w_l||_{\mathtt{HS}} \\
&\leq& \sum_{l=0}^{\infty}d_{l}\exp(M(L\lambda_{l}^{1/\nu}))||w_l||_{l^{1}(\mathbb{C}^{d_l})}\nonumber\\
&\leq& C\sum_{l=0}^{\infty}(1+\lambda_{l})^{\frac{n}{\nu}}\exp(-M(L\lambda_{l}^{1/\nu}))\exp(2M(L\lambda_{l}^{1/\nu}))||w_l||_{l^{1}(\mathbb{C}^{d_l})}\nonumber\\
&\leq& C^{\prime}\sum_{l=0}^{\infty}\sum_{j=1}^{d_l}\exp(2M(L\lambda_{l}^{1/\nu}))|(w_l)_j|\nonumber\\
&\leq& C^{\prime\prime} \sum_{l=0}^{\infty}\sum_{j=1}^{d_l}\exp(M(L_2\lambda_{l}^{1/\nu}))|(w_l)_j|<\infty,\nonumber
\end{eqnarray} 
where 
$L_2=L\sqrt{A} H,$ where $A,H>0$ in (M.2). The above claim will be true if we can show 
that  $\exp(2M(L\lambda_{l}^{1/\nu}))\leq\exp(M(L_2\lambda_{l}^{1/\nu})).$ This follows from property (M.2).

$\Longleftarrow$: 
Converse follows similarly using estimates \eqref{EQ:in}.
\end{proof}

We can now prove  Theorem \ref{P:perfect}.

\begin{proof}[Proof of Theorem \ref{P:perfect}]  
We always have $$\Gamma_{\{M_k\}}(X)\subseteq \left[\hat{\Gamma_{\{M_k\}}(X)}\right]^{\wedge}$$ from the definition.
So we need to prove that
$\left[\hat{\Gamma_{\{M_k\}}(X)}\right]^{\wedge}\subseteq \Gamma_{\{M_k\}}(X).$

Let $w=(w_l)_{l\in{\mathbb{N}}_{0}}\in \left[\hat{\Gamma_{\{M_k\}}(X)}\right]^{\wedge}$ , $w_{l}=\left(w_{l}^{j}\right)_{j=1}^{d_l},$ and define 
$$\phi(x):=\sum_{l=0}^{\infty}w_l \cdot e_{l}(x)=\sum_{l=0}^{\infty}\sum_{j=1}^{d_l} w^{j}_{l} e^{j}_{l}(x),$$
the series makes sense because of Lemma \ref{L:perfect1} and estimates \eqref{EQ:weyllaw}  and \eqref{EQ:estlambdas}.
Then we have
\begin{eqnarray} \hat{\phi}(m,k)&=&\left(\phi, e^{k}_{m}\right)_{L^2}\nonumber\\
&=& \int_{X} \phi(x)\bar{e^{k}_{m}(x)}dx\nonumber\\
&=&\sum_{l=0}^\infty\sum_{j=1}^{d_l}\int_{X} w^{j}_{l}e^{j}_{l}(x)\bar{e^{k}_{m}(x)}dx\nonumber\\
&=&\sum_{l=0}^\infty\sum_{j=1}^{d_l} w_{l}^{j}\delta_{lm}\delta_{jk}=w^{k}_{m}, ~~~~1\leq j\leq d_l, ~1\leq k\leq d_m.\end{eqnarray}
This gives $$||\hat{\phi}(m)||_{\mathtt{HS}}=||w_m||_{\mathtt{HS}}.$$
Now since $w\in \left[\hat{\Gamma_{\{M_k\}}(X)}\right]^{\wedge},$  by Lemma \ref{L:perfect1} there exists $L>0$ such that
$$\sum_{l=0}^{\infty}\exp\left(M(L\lambda^{1/\nu}_{l})\right)||w_l||_{\mathtt{HS}}<\infty.$$
Since $\hat{\phi}(l)=w_{l},$ it follows that there eists $C>0$ such that 
$$||\hat{\phi}(l)||_{\mathtt{HS}}\leq C\exp\left(-M(L\lambda^{1/\nu}_{l})\right).$$
By Theorem \ref{THM:gamma}, we have $\phi\in \Gamma_{\{M_k\}}(X).$ Hence we have shown that
$$\left[\hat{\Gamma_{\{M_k\}}(X)}\right]^{\wedge}\subseteq \Gamma_{\{M_k\}}(X).$$

So we have $\left[\hat{\Gamma_{\{M_k\}}(X)}\right]^{\wedge}=\Gamma_{\{M_k\}}(X),$ and  hence $\Gamma_{\{M_k\}}(X)$ is a perfect space. 
\end{proof}
Next we proceed to prove the equivalence of two expressions for the duality.
\begin{lem}
Let $v\in \Gamma_{\{M_k\}}(X)$ and  $w\in\widehat{\Gamma_{\{M_{k}\}}}(X)$, then 
$$\sum_{k=0}^{\infty}||(\hat{v}_{k})||_{\mathtt{HS}}||(w_{k})||_{\mathtt{HS}}<\infty$$ if and only if $$\sum_{k=0}^{\infty}\sum_{i=1}^{d_k}|(\hat{v_{k}})_i| |(w_k)_i|<\infty.$$
\end{lem}

\begin{proof}
$\Longrightarrow$: The proof is straightforward, following from the estimate
$$
\left(\sum_{i=1}^d a_i b_i\right)^2 \lesssim \left(\sum_{i=1}^d a_i^2\right) \left(\sum_{i=1}^d b_i^2\right).
$$

$\Longleftarrow$:
We will be using the equality
$$\left(\sum_{i=1}^{n}|a_i|\right)\left(\sum_{i=1}^n|b_i|\right)=\sum_{i=1}^{n}|a_i||b_i|+\sum_{i=1}^{n}|a_i|\left(\sum_{j=1}^n|b_j|-|b_i|\right)$$ for any $a_i,b_i\in\mathbb{R}$, 
yielding
\begin{eqnarray}\label{4.5}
&&||(\hat{v}_{k})||_{\mathtt{HS}}||(w_{k})||_{\mathtt{HS}}\nonumber\\
&\leq& \sum_{i=1}^{d_k}|(\hat{v}_{k})_i|\sum_{i=1}^{d_k} |(w_k)_i|\nonumber\\
&=& \sum_{i=1}^{d_k}|(\hat{v}_{k})_i| |(w_k)_i| +\sum_{i=1}^{d_k}|(w_k)_{i}|\left(\sum_{j=1}^{d_k}|(\hat{v}_k)_j|-|(\hat{v}_k)_i|\right).
\end{eqnarray}
We consider the second term in the above inequality, that is,
\begin{equation}\label{4.6}
\left(\sum_{i=1}^{d_k}|(w_k)_{i}|\left(\sum_{j=1}^{d_k}|(\hat{v}_k)_j|-|(\hat{v}_k)_i|\right)\right)\leq C\left(\sum_{i=1}^{d_k}|(w_{k})_i|(d_ke^{-M(L\lambda_{k}^{1/\nu})})\right),
\end{equation}
for some $C>0$ and $L>0.$
Then using \eqref{4.6} in \eqref{4.5} we get
\begin{equation}\label{4.7}
\sum_{k=0}^\infty||(\hat{v}_{k})||_{\mathtt{HS}}||(w_{k})||_{\mathtt{HS}}
\leq\sum_{k=0}^\infty \sum_{i=1}^{d_k}|(w_k)_i| \left(|(\hat{v}_{k})_i|+Cd_ke^{-M(L\lambda_{k}^{1/\nu})}\right).
\end{equation}
Now let $|{(\hat{u}_{k})_{i}}|=|(\hat{v}_{k})_i|+Cd_ke^{-M(L\lambda_{k}^{1/\nu})}$, for $i=1,2,...,d_{k},$ and $k\in\mathbb{N}_{0}.$
So then we have
$$\sum_{k=0}^\infty||(\hat{v}_{k})||_{\mathtt{HS}}||(w_{k})||_{\mathtt{HS}}\leq\sum_{k=0}^\infty \sum_{i=1}^{d_k}|(w_k)_i| |{(\hat{u}_{k})_{i}}|.$$
Now for some $C^{\prime\prime}>0$ and $L_{2}>0$, we have
\begin{eqnarray}
||{\hat{u}_{k}}||_{\mathtt{HS}}&=&\left(\sum_{i=1}^{d_k}\left(|(\hat{v}_{k})_i|^{2}+C^2d_k^{2}e^{-2M(L\lambda_{k}^{1/\nu})}+2Cd_k|(\hat{v}_{k})_i|e^{-M(L\lambda_{k}^{1/\nu})})\right)\right)^{1/2}\nonumber\\
&\leq& C^{\prime\prime}e^{-M(L_{2}\lambda_{k}^{1/\nu})},
\end{eqnarray}
i.e, $u\in \Gamma_{\{M_{k}\}}(X).$ 
This is true since
\begin{eqnarray}
&&\left(\sum_{i=1}^{d_k}\left(|(\hat{v}_{k})_i|^{2}+C^2d_k^{2}e^{-2M(L\lambda_{k}^{1/\nu})}+2Cd_k|(\hat{v}_{k})_i|e^{-M(L\lambda_{k}^{1/\nu})}\right)\right)^{1/2}\nonumber\\
&\leq& \left(\sum_{i=1}^{d_k}\left(C^2e^{-2M(L\lambda_{k}^{1/\nu})}+C^2 d_{k}^{2}e^{-2M(L\lambda_{k}^{1/\nu})}+2C^2 d_{k}e^{-2M(L\lambda_{k}^{1/\nu})}\right)\right)^{1/2}\nonumber\\
&\leq&\left(\sum_{i=1}^{d_k}C^2\left(1+d_k\right)^{2}e^{-2M(L\lambda_{k}^{1/\nu})}\right)^{1/2}\nonumber\\
&\leq& C(1+d_{k})^{3/2}e^{-M(L\lambda_{k}^{1/\nu})}\nonumber\\
&\leq& C^{\prime}e^{-\frac{1}{2}M(L\lambda_{k}^{1/\nu})}\nonumber\\
&\leq& C^{\prime\prime}e^{-M(L_{2}\lambda^{1/2})},\nonumber
\end{eqnarray}
where $L_{2}=\frac{L}{\sqrt A  H},$ with $A, H$ are constants in condition $(M.2).$

Now since $w\in\widehat{\Gamma_{\{M_{k}\}}}(X),$ so  from \eqref{4.7} we have
$$\sum_{k=0}^\infty||(\hat{v}_{k})||_{\mathtt{HS}}||(w_{k})||_{\mathtt{HS}}\lesssim \sum_{k=0}^\infty \sum_{i=1}^{d_k}|(w_k)_i| |(\hat{u}_{k})_i|<\infty,$$ completing the proof.
\end{proof}

\begin{thm}[Adjointness Theorem]\label{THM:adj}
Let $\{M_k\}$ and $\{N_k\}$ satisfy conditions $(M.{0})$-$(M.{3}).$ A linear mapping $f:\Gamma_{\{M_k\}}(X)\rightarrow \Gamma_{\{N_k\}}(X)$ is sequential  if and only if  $f$ is represented by an infinite tensor $(f_{kjli}), $ ~ $k,j\in \mathbb{N}_{0},$ $1\leq l\leq d_{j}$ and $1\leq i\leq d_k$ such that for any $u\in\Gamma_{\{M_k\}}(X)$ and $v\in\hat{\Gamma_{\{N_k\}}}(X)$ we have
\begin{equation} \label{EQ:f1}
\sum_{j=0}^{\infty}\sum_{l=1}^{d_j}|f_{kjli}||\hat{u}(j,l)|<\infty, ~~\textrm{for ~all}~k\in\mathbb{N}_{0}, ~i=1,2,...,d_k, 
\end{equation} and 
\begin{equation}\label{EQ:f2}
\sum_{k=0}^{\infty}\sum_{i=1}^{d_k}\left|(v_k)_i\right|\left|{\left(\sum_{j=0}^{\infty}f_{kj}\hat{u}(j)\right)_{i}}\right|<\infty.
\end{equation}
Furthermore, the adjoint mapping $\hat{f}:\hat{\Gamma_{\{N_k\}}}(X)\rightarrow \hat{\Gamma_{\{M_k\}}}(X)$ defined by the formula $\hat{f}(v)=v\circ f$ is also sequential, and the transposed  matrix ${(f_{kj})}^{t}$ represents $\hat{f}$, with $f$ and $\hat f$ related by $\langle f(u),v\rangle=\langle u,\hat f (v)\rangle.$
\end{thm}

Let us summarise the ranges for indices in the used notation as well as give more explanation to \eqref{EQ:f2}. 
For $f:\Gamma_{\{M_k\}}(X)\rightarrow \Gamma_{\{N_k\}}(X)$ and $u\in \Gamma_{\{M_k\}}(X)$ we write
\begin{equation}\label{EQ:notf}
\mathbb C^{d_k}\ni f(u)_k=\sum_{j=0}^{\infty}f_{kj}\widehat{u}(j)=
\sum_{j=0}^\infty \sum_{l=1}^{d_j} f_{kjl}\widehat{u}(j,l),\quad k\in\N_0,
\end{equation} 
so that
\begin{equation}\label{EQ:notf2}
f_{kjl}\in \mathbb C^{d_k},\;
f_{kjli}\in\mathbb C,\quad k,j\in\N_0,\; 1\leq l\leq d_j,\; 1\leq i \leq d_k,
\end{equation} 
and 
\begin{equation}\label{EQ:notf3}
\mathbb C\ni (f(u)_k)_i=f(u)_{ki} = \sum_{j=0}^\infty\sum_{l=1}^{d_j} f_{kjli}\widehat{u}(j,l),\quad k\in\N_0,\; 1\leq i \leq d_k,
\end{equation} 
where we view $f_{kj}$ as a matrix, $f_{kj}\in\mathbb{C}^{d_k\times d_j}$, and the product of the matrices has been explained in \eqref{EQ:notf}.

\begin{rem}
Let us now briefly describe how the tensor $(f_{kjli})$, $k,j\in \mathbb{N}_{0},$ $1\leq l\leq d_{j}$,  $1\leq i\leq d_k$, is constructed given a
sequential mapping $f:\Gamma_{\{M_k\}}(X)\rightarrow \Gamma_{\{N_k\}}(X)$.
For every $k\in \mathbb{N}_{0}$ and $1\leq i\leq d_k$, define the family
$v^{ki}=\left(v^{ki}_{j}\right)_{j\in\mathbb{N}_{0}}$ such that each $v^{ki}_{j}\in \C^{d_j}$ is defined by
\begin{equation}\label{EQ:defv}
 v^{ki}_{j}(l)=\begin{cases}
               1,~~~~~j=k, l=i,\\
               0,~~~~~ \textrm{otherwise}.
            \end{cases}
\end{equation} 
 Then  $v^{ki}\in \left[\Gamma_{\{N_{k}\}}(X)\right]^{\wedge}$, and since
$f$ is sequential we have  $v^{ki}\circ f\in\left[\Gamma_{\{M_{k}\}}(X)\right]^{\wedge}$, and we can write $v^{ki}\circ f=\left(v^{ki}\circ f\right)_{j\in\N_0},$ where $(v^{ki}\circ f)_{j}\in\C^{d_j}.$  
Then for each $1\leq l \leq d_j$ we set
\begin{equation}\label{EQ:deff}
f_{kjli}:=(v^{ki}\circ f)_{j}(l),
\end{equation} 
the $l^{th}$ component of the vector $(v^{ki}\circ f)_{j}\in\C^{d_j}.$
The formula \eqref{EQ:deff} will be shown in the proof of Theorem {\ref{THM:adj}}.
In particular, since for $\phi\in\Gamma_{\{M_{k}\}}(X)$ we have $f(\phi)\in\Gamma_{\{N_{k}\}}(X),$
it will be a consequence of \eqref{EQ: 4.27} and \eqref{EQ: 4.28} later on that
\begin{equation}
\label{EQ:deff2}
v^{ki}\circ f(\phi)=(\widehat{f(\phi)})(k,i)=\sum_{j=0}^{\infty}\sum_{l=1}^{d_j}f_{kjli}\hat{\phi}(j,l),
\end{equation}
so that the tensor $(f_{kjli})$ is describing the transformation of the Fourier coefficients of $\phi$ into those of $f(\phi)$.

Another meaning of condition \eqref{EQ:f1} is that if for each $k\in \mathbb{N}_{0}$ and $1\leq i\leq d_k$ we define
$$
f^{ki}(j,l):=f_{kjli},
$$
then $f^{ki}\in \left[\Gamma_{\{M_{k}\}}(X)\right]^{\wedge}$. Condition \eqref{EQ:f2} is the continuity condition saying that for every  $u\in\Gamma_{\{M_{k}\}}(X)$ we have that
$$
\sum_{j=0}^{\infty}\sum_{l=1}^{d_j}f_{kjli}\hat{u}(j,l)\in \Gamma_{\{N_{k}\}}(X).
$$
\end{rem}

To prove Theorem {\ref{THM:adj}} we first establish the following lemma.

\begin{lem} \label{L:L1} 
Let  $(f_{kjli})_{k,j\in{\N_{0}}, 1\leq l\leq d_j, 1\leq i\leq d_k}$ be an infinite tensor  satisfying \eqref{EQ:f1} and  \eqref{EQ:f2}. Then for any $u\in\Gamma_{\{M_k\}}(X)$ and $v\in \left[\Gamma_{\{N_k\}}(X)\right]^{\wedge},$ we have
$$\lim_{n\rightarrow\infty}\sum_{k=0}^{\infty}\sum_{i=1}^{d_k}\left|(v_k)_i\right|\left|{\left(\sum_{0\leq j\leq n}f_{kj}\hat{u}(j)\right)_{i}}\right|=\sum_{k=0}^{\infty}\sum_{i=1}^{d_k}\left|(v_k)_i\right|\left|{\left(\sum_{j=0}^{\infty}f_{kj}\hat{u}(j)\right)_{i}}\right|.$$
\end{lem}

\begin{proof}[Proof of Lemma \mbox{\ref{L:L1}}]

 Let $u\in \Gamma_{\{M_k\}}(X)$ and $u\approx \left(\hat u(l)\right)_{l\in {\mathbb{N}}_{0}}.$ Define $u^{n}:=\left(\hat u ^{(n)}(l)\right)_{l\in {\mathbb{N}}_{0}}$  by setting  
\begin{equation}
\hat{u}^{(n)}(l)=\begin{cases}
\hat u(l), \; l\leq n, \nonumber\\
 0, \quad\;\, l>n.\end{cases}\nonumber
\end{equation} 
 Then for any  $w\in \hat{\Gamma_{\{M_k\}}}(X),$  $\langle u-u^{n}, w\rangle\rightarrow 0$ as $n\rightarrow \infty.$ This is true since $\sum_{l=0}^{\infty}\left|\hat{u}(l)\cdot w_l \right|<\infty$ so that  
 $$\left| \langle u-u^{n}, w\rangle \right|\leq \sum_{l\geq n}\left|\hat{u_l}\cdot w_{l}\right|\rightarrow {0}$$ as $n\rightarrow \infty.$ 
  Now for any $u\in \Gamma_{\{M_k\}}(X)$  and $v\in \left[\Gamma_{\{N_k\}}(X)\right]^{\wedge} $ and from \eqref{EQ:f1} and \eqref{EQ:f2} we have 
   \begin{multline}\label{EQ:long}
  \langle f(u), v \rangle=\sum_{k=0}^{\infty} \left(f(u)\right)_{k}\cdot v_{k}
=\sum_{k=0}^{\infty}\left(\sum_{j=0}^{\infty}f_{kj}\hat{u}(j)\right) \cdot v_{k}
\\
=\sum_{k=0}^{\infty}\sum_{j=0}^\infty\sum_{\ell=1}^{d_j}\sum_{i=1}^{d_k} f_{kj\ell i}\widehat{u}(j,\ell)(v_k)_{i}
=\sum_{j=0}^{\infty}\sum_{\ell=1}^{d_j}\widehat{u}(j,\ell) \sum_{k=0}^{\infty}\sum_{i=1}^{d_k}
f_{kj\ell i}(v_k)_i
\\
=\sum_{j=0}^{\infty}\sum_{\ell=1}^{d_j}\widehat{u}(j,\ell) \sum_{k=0}^{\infty}
f_{kj\ell }\cdot v_k
=\sum_{j=0}^\infty\hat{u}(j)\cdot (v\circ f)_{j}=\langle u, v\circ f\rangle,
  \end{multline}
  where  
 $$\mathbb C^{d_j}\ni (v\circ f)_{j}=\left\{\sum_{k=0}^{\infty}
f_{kj\ell }\cdot v_k\right\}_{\ell=1}^{d_j},\quad j\in \mathbb{N}_{0},$$ and 
$$v\circ f=\left\{(v\circ f)_{j}\right\}_{j=0}^{\infty}.$$

For sequential mapping $f$, $v\circ f\in\left[\Gamma_{\{M_k\}}(X)\right]^{\wedge} $ and 
$$\sum_{j=0}^{\infty}u(j)\cdot(v\circ f)_j=\langle u, v\circ f\rangle=\left(v\circ f\right)(u),$$ so that we can write $\left(v\circ f\right)\in \mathbb{C}^{d_j}$and also $\left(v\circ f\right)(u)=\langle v, f(u)\rangle$. So for any $u\approx(\hat{u}(j))_{j\in\N_0}\in \Gamma_{\{M_k\}}(X) $  from the definition of $\left[\Gamma_{\{M_k\}}(X)\right]^{\wedge}$ we have $$
\sum_{j\in \N_0}\sum_{l=1}^{d_j}\left|(v\circ f)_{jl}\right|\left|\hat{u}(j,l)\right|<\infty.$$  Hence the series  $\sum_{j=0}^{\infty}\left|(v\circ f)_{j}\cdot\hat{u}(j)\right|$ is convergent.

We can then conclude that $v\circ f\in \left[\Gamma_{\{M_k\}}(X)\right]^{\wedge} $ and we have
$$\langle f(u)-f(u^{n}), v\rangle=\langle u-u^{n}, v\circ f\rangle\rightarrow 0$$ as $n\rightarrow \infty.$ Therefore,
$$\langle f(u), v\rangle=\lim_{n\rightarrow\infty}\langle f(u^{n}), v\rangle,$$ for all $u\in \Gamma_{\{M_{k}\}}(X)$ and $v\in[\Gamma_{\{N_k\}}(X)]^{\wedge}.$
Hence for any $u\in \Gamma_{\{M_k\}}(X)$ and $v\in \left[\Gamma_{\{N_k\}}(X)\right]^{\wedge}$ we have
$$\lim_{n\rightarrow\infty}\sum_{k=0}^{\infty}v_k\cdot\left(\sum_{0\leq j\leq n}f_{kj}\hat{u}(j)\right) =\sum_{k=0}^{\infty}v_k\cdot\left(\sum_{j=0}^{\infty}f_{kj}\hat{u}(j)\right),$$
that is,
$$\lim_{n\rightarrow\infty}\sum_{k=0}^{\infty}\sum_{i=1}^{d_k}(v_k)_i\left(\sum_{j\leq n}f_{kj}\hat{u}(j)\right)_i =\sum_{k=0}^{\infty}\sum_{i=1}^{d_k}(v_k)_{i}\left(\sum_{j=0}^{\infty}f_{kj}\hat{u}(j)\right)_i.$$ Now we will use the fact that if $u\in \Gamma_{\{M_k\}}(X)$ then $|u|\in\Gamma_{\{M_k\}}(X)$ where $|u|=\left(\hat{ |u}|_j\right)_{j\in\mathbb{N}_{0}},$ $\hat{ |u|}_{j}\in \mathbb{R}^{d_j},$  with
\begin{align}
    \hat {|u|}_{j} &:= \begin{bmatrix}
          |\hat{ u}(j,1)| \\
          |\hat{ u}(j,2)|  \\
           \vdots \\
          |\hat {u}(j,d_j)| 
         \end{bmatrix},\nonumber
  \end{align} in view of the Theorem \ref{P:perfect}. The 
 same is true for the dual space $\left[\Gamma_{\{N_k\}}(X)\right]^{\wedge}.$ 
 So then this argument gives
$$\lim_{n\rightarrow\infty}\sum_{k=0}^{\infty}\sum_{i=1}^{d_k}\left|(v_k)_i\right|\left|{\left(\sum_{0\leq j\leq n}f_{kj}\hat{u}(j)\right)_{i}}\right|=\sum_{k=0}^{\infty}\sum_{i=1}^{d_k}\left|(v_k)_i\right|\left|{\left(\sum_{j=0}^{\infty}f_{kj}\hat{u}(j)\right)_{i}}\right|.$$ The proof is complete.
\end{proof}

 \begin{proof}[Proof of Theorem \mbox{\ref{THM:adj}}]  
 Let us assume first that the mapping $f:\Gamma_{\{M_k\}}(X)\rightarrow\Gamma_{\{N_{k}\}}(X)$ can be represented by $f=(f_{kjli})_{k,j\in\mathbb{N}_{0},
 1\leq l\leq d_j, 1\leq i\leq d_k},$ an infinite tensor such that\begin{equation}\sum_{j=0}^{\infty}\sum_{l=1}^{d_j}|f_{kjli}||\hat{u}(j,l)|<\infty, ~~\textrm{for ~all}~k\in\mathbb{N}_{0},~ i=1,2,\ldots,d_k,\end{equation}  and
 \begin{equation}\sum_{k=0}^{\infty}\sum_{i=1}^{d_k}\left|(v_k)_i\right|\left|{\left(\sum_{j=0}^{\infty}f_{kj}\hat{u}(j)\right)_{i}}\right|<\infty\end{equation} hold for all $u\in\Gamma_{\{M_k\}}(X)$ and $v\in[\Gamma_{\{N_k\}}(X)]^{\wedge}.$
 
 Let $\hat{u}_{1}=\left(\hat{u_1}(p)\right)_{p\in{\mathbb{N}}_{0}}$ be such that for some $j,l$ where $j\in{\mathbb{N}}_{0},$ $1\leq l\leq d_j$  we have 
 \begin{equation} \hat{u_1}(p,q)=\begin{cases}
 {1},~~~~ p=j, \; q=l,\\
 0, ~~~~~~\textrm{otherwise}.\end{cases}\nonumber\end{equation}
 
 Then $u_{1}\in \Gamma_{\{M_k\}}(X)$  so $fu_1=f(u_{1})\in\Gamma_{\{N_k\}}(X)$ and
 \begin{eqnarray}\label{EQ:4.21}
 \left(fu_1\right)_{k}&=&\sum_{p=0}^{\infty}f_{kp}\hat{u}_{1}(p)\nonumber\\
 &=&\sum_{p=0}^{\infty}\sum_{q=1}^{d_p}f_{kpq}\hat{u_{1}}(p,q)\nonumber\\
 &=& \sum_{q=1}^{d_j}f_{kjq}\hat{u_1}(j,q)\nonumber\\
 &=&f_{kjl}\in \C^{d_k}.
 \end{eqnarray}

 We now first show that
 $$\widehat{\left(fu\right)}(k)=\sum_{j=0}^{\infty}\sum_{l=1}^{d_j}f_{kjl}\hat{u}(j,l),$$ where $f_{kjli}\in \mathbb{C}$ for each $k,j\in\mathbb{C},$ $1\leq l \leq d_j$ and $1\leq i\leq d_k.$
The way in which $f$ has been defined we have
 $$(fu)_{k}= \sum_{j=0}^{\infty}\sum_{l=1}^{d_j}f_{kjl}\hat{u}(j,l),\quad f_{kjl}\in \C^{d_k}.$$
Also since $u\in\Gamma_{\{M_k\}}(X)$, from our assumption we have $fu\in\Gamma_{\{N_k\}}(X)$  and $fu\approx \left(\hat{(fu)}(j)\right)_{j\in\mathbb{N}_{0}}$so $(fu)_{k}\approx\hat{(fu)}(k).$

We  can then write $\hat{(fu)}(k)=\sum_{j}\sum_{l=1}^{d_j}f_{kjl}\hat{u}(j,l).$
Since we know that $v\in{\left[\Gamma_{\{N_k\}}(X)\right]^{\wedge}}$ and $fu\in\Gamma_{\{N_k\}}(X),$ we have
$$\sum_{k=0}^{\infty}\sum_{i=1}^{d_k}|(v_k)_i||(\widehat{(fu)}(k))_i|=\sum_{k=0}^{\infty}\sum_{i=1}^{d_k}|(v_k)_i||\sum_{j=0}^{\infty}\sum_{l=1}^{d_j} f_{kjli}\hat{u}(j,l)|<\infty.$$
 In particular  using the definition of $u_1$ and \eqref{EQ:4.21} we get
\begin{eqnarray}\label{EQ:4.22}\sum_{k=0}^{\infty}\sum_{i=1}^{d_k}|(v_k)_i|\left|\sum_{p=0}^{\infty} \sum_{q=1}^{d_k}f_{kpqi}\hat{u_1}(p,q)\right|=\sum_{k=0}^{\infty}\sum_{i=1}^{d_k}|(v_k)_i||f_{kjli}|<\infty,\end{eqnarray}
for any $j\in \mathbb{N}_{0}$ and $1\leq l\leq d_j.$\\
Now for any $u\in \Gamma_{\{M_k\}}(X)$ consider 
$$J=\sum_{j=0}^{\infty}\sum_{l=1}^{d_j}|\sum_{k=0}^{\infty}\sum_{i=1}^{d_k}(v_{k})_{i}f_{kjli}| |\hat{u}(j,l)|.$$
Then we consider the series 
$$I_{n}:=\sum_{j\leq n}\sum_{l=1}^{d_j}|\sum_{k=0}^{\infty}\sum_{i=1}^{d_k}(v_{k})_{i}f_{kjli}| |\hat{u}(j,l)|,$$
so that we have
\begin{eqnarray}
I_{n}&=&\sum_{j\leq n}\sum_{l=1}^{d_j}|\sum_{k=0}^{\infty}\sum_{i=1}^{d_k}(v_{k})_{i}f_{kjli}| |\hat{u}(j,l)|\nonumber\\
&=&\sum_{j\leq n}\sum_{l=1}^{d_j}|\sum_{k=0}^{\infty}\sum_{i=1}^{d_k}(v_{k})_{i}f_{kjli}\hat{u}(j,l)|.\nonumber
\end{eqnarray}
Let $\epsilon=(\epsilon_i)_{1\leq i\leq d_{k}},$ $k\in{\mathbb{N}}_{0}$, be  such that $\epsilon_i\in\mathbb{C}$ and $|\epsilon_i|=1, $ for  all $i$ and  such that
$$|\sum_{k=0}^{\infty}\sum_{i=1}^{d_k}(v_{k})_{i}f_{kjli})\hat{u}(j,l)|=\sum_{k=0}^{\infty}\sum_{i=1}^{d_k}(v_{k})_{i}f_{kjli}\hat{u}(j,l)\epsilon_{i}.$$
Then 
\begin{eqnarray}
I_n &=&\sum_{j\leq n}\sum_{l=1}^{d_j}\sum_{k=0}^{\infty}\sum_{i=1}^{d_k}(v_{k})_{i}f_{kjli}\hat{u}(j,l)\epsilon_{i}\nonumber\\
&\leq & \sum_{k=0}^{\infty}\sum_{i=1}^{d_k}|(v_{k})_{i}|\left|\sum_{j\leq n}\sum_{l=1}^{d_j}f_{kjli})\hat{u}(j,l)\epsilon_{i} \right|.
\end{eqnarray}
It follows from Lemma \ref{L:L1} that
$$\lim_{n\rightarrow\infty}\sum_{k=0}^{\infty}\sum_{i=1}^{d_k}|(v_{k})_{i}|\left|\sum_{j\leq n}\sum_{l=1}^{d_j}f_{kjli}\hat{u}(j,l)\epsilon_{i} \right| =  \sum_{k=0}^{\infty}\sum_{i=1}^{d_k}|(v_{k})_{i}|\left|\sum_{j=0}^{\infty}\sum_{l=1}^{d_j}f_{kjli}\hat{u}(j,l)\epsilon_{i} \right|<\infty.$$

Then \begin{equation}\label{EQ:4.24} J=\sum_{j=0}^{\infty}\sum_{l=1}^{d_j}|\sum_{k=0}^{\infty}\sum_{i=1}^{d_k}(v_{k})_{i}f_{kjli}| |\hat{u}(j,l)|<\infty.\end{equation}

So  we proved that if $(f_{kjli})$ satistfies
\begin{itemize}
\item $\sum_{j=0}^{\infty}\sum_{l=1}^{d_j}|f_{kjli}||\hat{u}(j,l)|<\infty$,
\item $\sum_{k=0}^{\infty}\sum_{i=1}^{d_k}\left|(v_k)_i\right|\left|{\left(\sum_{j=0}^{\infty}f_{kj}\hat{u}(j)\right)_{i}}\right|<\infty$,
\end{itemize}
 then  for any $u\in{\Gamma_{\{M_{k}\}}(X)}$ and $v\in\left[\Gamma_{\{N_{k}\}}(X)\right]^{\wedge}$ we have from \eqref{EQ:4.22} and \eqref{EQ:4.24} respectively, that is,
 \begin{enumerate}
 \item $\sum_{k=0}^{\infty}\sum_{i=1}^{d_k}|(v_k)_i||f_{kjli}|<\infty$, 
 \item $\sum_{j=0}^{\infty}\sum_{l=1}^{d_j}|\sum_{k=0}^{\infty}\sum_{i=1}^{d_k}(v_{k})_{i}f_{kjli})| |\hat{u}(j,l)|<\infty.$
 \end{enumerate}
Now recall that for $f: \Gamma_{\{M_{k}\}}(X)\rightarrow \Gamma_{\{N_{k}\}}(X)$ we have
$$(f(u))_{k}=\sum_{j=0}^{\infty}\sum_{l=1}^{d_j}f_{kjl}\hat{u}(j,l),$$ 
for any $u\in{\Gamma_{\{M_{k}\}}(X)},$ then for any  $v\in\left[\Gamma_{\{N_{k}\}}(X)\right]^{\wedge}$, the composed mapping  $v\circ f: \Gamma_{\{M_{k}\}}(X)\rightarrow \mathbb{C}$ is given by
\begin{eqnarray}(v\circ f)(u)&=&\sum_{k=0}^{\infty}v_k\cdot (f(u))_k=\sum_{k=0}^{\infty}\sum_{i=1}^{d_k}(v_{k})_{i}\left(\sum_{j=0}^{\infty} \sum_{l=1}^{d_j}f_{kjli}\hat u(j,l)\right)\nonumber\\
&=&\sum_{j=0}^{\infty}\sum_{l=1}^{d_j}\left(\sum_{k=0}^{\infty} \sum_{i=1}^{d_k}(v_{k})_{i}f_{kjli}\right)\hat u(j,l).\end{eqnarray}

So by (ii) we get that
$$\left|(v\circ f)(u)\right|\leq\sum_{j=0}^{\infty}\sum_{l=1}^{d_j}|\sum_{k=0}^{\infty}\sum_{i=1}^{d_k}(v_{k})_{i}f_{kjli}| |\hat{u}(j,l)|<\infty.$$
So $\hat f(v)=(\hat f (v)_{j l})_{j\in \mathbb N, 1\leq l\leq d_j},$ with $\hat{f}(v)_{jl}= \sum_{k=0}^{\infty}\sum_{i=1}^{d_k}(v_{k})_{i}f_{kjli}\in [\Gamma_{\{M_{k}\}}(X)]^{\wedge}$ (from the definition of $ [\Gamma_{\{M_{k}\}}(X)]^{\wedge}$), that is $f$ is sequential.
And then $\langle f(u), v\rangle=\langle u,\hat f(v)\rangle$  is also true.

\medskip
Now to prove the converse part we assume that $f:\Gamma_{\{M_k\}}(X)\rightarrow \Gamma_{\{N_k\}}(X)$ is sequential. We have to show that $f$ can be represented as  $f\approx(f_{kjli})_{k,j\in\mathbb{N}_{0},1\leq l\leq d_j, 1\leq i\leq d_k}$ and satisfies \eqref{EQ:f1} and \eqref{EQ:f2}.

Define for  $k,i$ where $k\in \mathbb{N}_{0}$ and $1\leq i\leq d_k,$  the sequence $u^{ki}=\left(u^{ki}_{j}\right)_{j\in\mathbb{N}_{0}}$ such that $u^{ki}_{j}\in \C^{d_j}$ and $u^{ki}_j(l)=\hat{u^{ki}}({j,l})$, given by 
\[
  u^{ki}_{j}(l)=\hat{u^{ki}}({j,l})=\begin{cases}
               1,~~~~~j=k, l=i,\\
               0,~~~~~ \textrm{otherwise}.
            \end{cases}
\]

Then  $u^{ki}\in \left[\Gamma_{\{N_{k}\}}(X)\right]^{\wedge}.$
Now since  $f$ is sequential we have  $u^{ki}\circ f\in\left[\Gamma_{\{M_{k}\}}(X)\right]^{\wedge}$ and $u^{ki}\circ f=\left(u^{ki}\circ f\right)_{j\in\N_0},$ where $(u^{ki}\circ f)_{j}\in\C^{d_j}.$  We denote $u^{ki}\circ f=\left(f^{ki}_{j}\right)_{j\in{\N_{0}}},$ where $f^{ki}_j=(u^{ki}\circ f)_{j}.$ Then $(f^{ki}_{j})_{j\in\N_{0}}\in \left[\Gamma_{\{M_{k}\}}(X)\right]^{\wedge}$ and $f^{ki}_{j}\in\C^{d_j}.$

Then for any $\phi\approx \left(\hat{\phi}(j)\right)_{j\in\N_0}\in \Gamma_{\{M_k\}}(X)$ we have
\begin{equation}\sum_{j=0}^{\infty}\sum_{l=1}^{d_j}|f^{ki}_{jl}||\hat{\phi}(j,i)|<\infty.\end{equation}
 
For $\phi\in\Gamma_{\{M_{k}\}}(X)$ we can write $f(\phi)\in\Gamma_{\{N_{k}\}}(X).$ We can write $$f(\phi)=\left((f(\phi))^{\wedge}(p)\right)_{p\in\N_{0}}.$$ So
\begin{eqnarray}\label{EQ: 4.27}
u^{ki}\circ f(\phi)&=&\sum_{j=0}^{\infty}u^{ki}_j\widehat{(f(\phi))}_j\nonumber\\
&=&\sum_{j=0}^{\infty}\sum_{l=1}^{d_j}u^{ki}_{jl}\widehat{(f(\phi))}(j,l)\nonumber\\
&=&(\widehat{f(\phi)})(k,i)~(\textrm{from~ the ~definition~ of~} u^{ki}).\end{eqnarray}

We have  $u^{ki}\circ f=(f^{ki})_{j}\in \left[\Gamma_{\{M_k\}(X)}\right]^{\wedge},$ so 
\begin{eqnarray}\label{EQ: 4.28}
(u^{ki}\circ f)(\phi)
&=&\sum_{j=0}^{\infty}f^{ki}_{j}\hat{\phi}(j)\nonumber\\
&=&\sum_{j=0}^{\infty}\sum_{l=1}^{d_j}f^{ki}_{jl}\hat{\phi}(j,l).
\end{eqnarray}
From \eqref{EQ: 4.27} and \eqref{EQ: 4.28} we have $(\widehat{f(\phi)})(k,i)=\sum_{j=0}^{\infty}\sum_{l=1}^{d_j}f^{ki}_{jl}\hat{\phi}(j,l).$\\
Hence  $(f(\phi))_{ki}=\sum_{j=0}^{\infty}\sum_{l=1}^{d_j}f^{ki}_{jl}\hat{\phi}(j,l),~~k\in\mathbb{N}_{0},$ and $1\leq i\leq d_k,$ that is  $f$ is represented by the tensor $\left\{(f^{ki}_{jl})\right\}_{k,j\in\N_0, 1\leq i\leq d_k, 1\leq l\leq d_j}$.\\
If we denote $f^{ki}_{jl}$ by $f^{ki}_{jl}= f_{kjli},$ we can say that $f$ is represented by the tensor $(f_{kjli})_{k,j\in\N_0, 1\leq l\leq d_j, 1\leq i\leq d_k}.$ 
Also let $v\in\hat{\left[\Gamma_{\{N_k\}}(X)\right]}.$ Since $f(\phi)\in \Gamma_{\{N_k\}}(X)$ for $\phi\in \Gamma_{\{M_k\}}(X),$ then from the definition of $\hat{\left[\Gamma_{\{N_k\}}(X)\right]}$  we have
$$\sum_{k=0}^{\infty}\sum_{i=1}^{d_k}|(v_k)_i|\sum_{j=0}^{\infty}\sum_{l=1}^{d_j}f_{kjli}\hat{\phi}(j,l)|<\infty.$$
This completes the proof of Theorem \ref{THM:adj}.
\end{proof}

\section{Beurling class of ultradifferentiable functions and ultradistributions}
\label{SEC:Beurling}

In this section we briefly summarise the counterparts of the results of the previous section for the case of Komatsu classes of Beurling type ultradifferentiable functions and ultradistributions.
For more details we refer to  \cite{DaR2} for a more detailed description of these spaces as well as of their duals and $\alpha$-duals in the sense of K\"othe.

 The class $\Gamma_{(M_k)}(X)$ is the space of $C^{\infty}$ functions $\phi$ on $X$ such that for every $h>0$ there exists $C_{h}>0$ such that we have
 \begin{equation}
 ||E^{k}\phi||_{L^2(X)}\leq C_{h}h^{\nu k}M_{\nu k}, ~k=0,1,2,...
\end{equation}
For these spaces, we replace condition (M.3) by condition

\medskip
\noindent
(M.3$'$) \quad for every $ l>0$ there exists $ C_{l}>0 $ such that 
$k!\leq C_{l} l^{k}M_{k}$, for all $k\in\mathbb{N}_{0}.$

\medskip

The counterpart of   \cite[Theorem \ref{THM:gamma} and Theorem \ref{THM:gammahat}]{DaR2}, holds for this class as well, namely, we have

\begin{thm}
\label{THM: Beurling 1}
Assume conditions (M.0), (M.1), (M.2), (M.$3'$). We have $\phi\in\Gamma_{(M_k)}(X)$ if and only if for every $L>0$ there exists $C_L>0$ such that
$$||\hat{\phi}(l)||_{\mathtt{HS}}\leq C_L\exp\{-M(L\lambda_{l}^{1/\nu})\}, \quad \textrm{ for all } l\geq 1.$$
For the dual space and for the $\alpha$-dual, the following statements are equivalent:
\begin{enumerate}
\item $v\in\Gamma^{\prime}_{(M_k)}(X);$
\item $v\in \left[\Gamma_{(M_k)}(X)\right]^{\wedge}$;
\item there exists $L>0$ such that we have
$$\sum_{l=0}^{\infty}\exp\left(-M(L\lambda_{l}^{1/\nu})\right)||v_{l}||_{\mathtt{HS}}<\infty;$$
\item there exists $L>0$ and $K>0$ such that $$||v_l||_{\mathtt{HS}}\leq K\exp\left(M(L\lambda_{l}^{1/\nu})\right)$$ holds for all $l\in\mathbb{N}_{0}.$ 
\end{enumerate}
\end{thm}
Again we note that given this characterisation of $\alpha$-duals, one can prove that they are perfect, in a way similar to the proof of Theorem \ref{P:perfect}, namely,
that 
\begin{equation}
\left[\Gamma_{(M_k)}(X)\right]=\left(\left[\Gamma_{(M_k)}(X)\right]^{\wedge}\right)^{\wedge}.
\end{equation}
Finally we can state the counterpart of the adjointness Theorem \ref{THM:adj}. 

\begin{thm}[Adjointness Theorem Beurling Case]
\label{THM:AdjB}
Let $\{M_k\}$ and $\{N_k\}$ satisfy conditions (M.{0})--(M.$3'$). A linear mapping $f:\Gamma_{(M_k)}(X)\rightarrow \Gamma_{(N_k)}(X)$ is sequential  if and only if  $f$ is represented by an infinite tensor $(f_{kjli}), $ ~ $k,j\in \mathbb{N}_{0},$ $1\leq l\leq d_{j}$ and $1\leq i\leq d_k$ such that for any $u\in\Gamma_{(M_k)}(X)$ and $v\in\hat{\Gamma_{(N_k)}}(X)$ we have
\begin{equation} \label{EQ:f11}
\sum_{j=0}^{\infty}\sum_{l=1}^{d_j}|f_{kjli}||\hat{u}(j,l)|<\infty, ~~\textrm{for ~all}~k\in\mathbb{N}_{0}, ~i=1,2,...,d_k, 
\end{equation} and 
\begin{equation}\label{EQ:f21}
\sum_{k=0}^{\infty}\sum_{i=1}^{d_k}\left|(v_k)_i\right|\left|{\left(\sum_{j=0}^{\infty}f_{kj}\hat{u}(j)\right)_{i}}\right|<\infty.
\end{equation}
Furthermore, the adjoint mapping $\hat{f}:\hat{\Gamma_{(N_k)}}(X)\rightarrow \hat{\Gamma_{(M_k)}}(X)$ defined by the formula $\hat{f}(v)=v\circ f$ is also sequential, and the transposed  matrix ${(f_{kj})}^{t}$ represents $\hat{f}$, with $f$ and $\hat f$ related by $\langle f(u),v\rangle=\langle u,\hat f (v)\rangle.$
\end{thm}
The proof of Theorem \ref{THM:AdjB} is similar to the corresponding proof in Theorem \ref{THM:adj} for the spaces $\Gamma_{\{M_{k}\}}(X),$ so we omit the repetition.

\end{document}